\documentclass[12pt]{amsart}

\usepackage{amsmath,
amsthm,amsfonts,amssymb, latexsym,enumerate,url}

\usepackage{tikz}

\usepackage[colorlinks=true, 
linkcolor = blue,
citecolor=black,
urlcolor=black,
]{hyperref}

\usepackage[backend=biber,
sorting=nyt,
style=numeric,
sortcites=true,
maxnames=99,
minnames=99,
]{biblatex}
\DeclareFieldFormat{prefixnumber}{\mkbibbold{#1}}
\DeclareFieldFormat{labelnumber}{\mkbibbold{#1}}
\DeclareFieldFormat[article]{number}{\mkbibparens{#1}}
\DeclareFieldFormat{addendum}{\mkbibparens{#1}}

\renewbibmacro*{note+pages}{%
  \printfield{note}
  \setunit{\addspace}
  \printfield{pages}
}

\DeclareFieldFormat{pages}{#1} 
\DeclareFieldFormat{vol}{#1}
\renewbibmacro{in:}{}

\DeclareFieldFormat[article]{title}{\textnormal{#1}}  

\DeclareFieldFormat[article]{volume}{\textbf{#1}}

\makeatletter
\renewenvironment{proof}[1][\proofname]{\par\pushQED{\qed}%
	\normalfont \topsep6\p@\@plus6\p@\relax
	\trivlist
	\item\relax
		{\bfseries
	#1\@addpunct{.}}\hspace\labelsep\ignorespaces}{%
	\popQED\endtrivlist\@endpefalse
}
\makeatother

\renewcommand{\bf}{\textbf}
\renewcommand{\it}{\textit}
\newcommand{\bb}[1]{\mathbb{#1}}
\newcommand{\te}[1]{\text{#1}}
\newcommand{\ce}[1]{\mathcal{#1}}

\newcommand{\Q}{\bb{Q}}

\newcommand{\X}{\mathcal{X}}
\newcommand{\Y}{\mathcal{Y}}

\renewcommand{\bar}[1]{\mkern 1.5mu\overline{\mkern-1.5mu#1\mkern-1.5mu}\mkern 1.5mu}

\newcommand{\ls}{<}
\newcommand{\g}{>}

\newcommand{\OR}{\mathcal{O}}

\newcommand{\nr}[1][G]{\textnormal{\bf{N}}_{#1}}
\newcommand{\cn}[1][G]{\textnormal{\bf{C}}_{#1}}
\newcommand{\op}{\textnormal{\bf{O}}}
\newcommand{\zn}{\textnormal{\bf{Z}}}

\newcommand{\IRR}{\textnormal{Irr}}

\newcommand{\GAL}{\textnormal{Gal}}

\newcommand{\XT}{\textnormal{Ext}}
\newcommand{\HD}{\textnormal{H}}
\newcommand{\LIN}{\textnormal{Lin}}

\renewcommand{\subset}{\subseteq}
\renewcommand{\supset}{\supseteq}

\begin{document}

\theoremstyle{plain}

\newtheorem{thm}{Theorem}[section]
\newtheorem{lem}[thm]{Lemma}
\newtheorem{conj}[thm]{Conjecture}
\newtheorem{pro}[thm]{Proposition}
\newtheorem{cor}[thm]{Corollary}
\newtheorem{que}[thm]{Question}
\newtheorem{rem}[thm]{Remark}
\newtheorem{defi}[thm]{Definition}
\newtheorem{hyp}[thm]{Hypothesis}

\newtheorem*{thmA}{THEOREM A}
\newtheorem*{thmB}{THEOREM B}
\newtheorem*{corC}{COROLLARY C}

\newtheorem*{thmC}{THEOREM C}
\newtheorem*{conjA}{CONJECTURE A}
\newtheorem*{conjB}{CONJECTURE B}
\newtheorem*{conjC}{CONJECTURE C}

\newtheorem*{thmAcl}{Main Theorem$^{*}$}
\newtheorem*{thmBcl}{Theorem B$^{*}$}

\numberwithin{equation}{section}

\marginparsep-0.5cm

\renewcommand{\thefootnote}{\fnsymbol{footnote}}
\footnotesep6.5pt

\title{Above Isaacs' Head Characters}

\author{Asier Arranz }
\address{Departamento de Matemáticas Fundamentales, UNED, 28040, Madrid, Spain}
\email{aarranz97@alumno.uned.es}




\begin{abstract}
The proof of the inductive McKay condition has been shown to imply that the character theory above the characters of degree not divisible by $p$ of a normal subgroup is locally determined. In this note, we establish a similar result for the Isaacs' head characters of a normal solvable subgroup of an arbitrary group. In particular, we give a new lower bound of the number of conjugacy classes of a finite group in terms of the Carter subgroups of any of its normal solvable subgroups. 
\end{abstract}

\thanks{This research was supported by a \it{Beca de Colaboración} from the Ministerio de Educación, Formación Profesional y Deportes (Spain), awarded for collaboration with the Departamento de Matemáticas Fundamentales at UNED. This work is part of my PhD thesis directed by G. Navarro, whom I thank for their guidance and useful comments. I also thank the anonymous referee for their meticulous reading and for the constructive comments that have significantly improved the presentation.}

\maketitle

\section{Introduction} 

One of the, perhaps unexpected, consequences of the proof \cite{B25} of the inductive McKay condition \cite{IMN} is that
the character theory above characters of $p'$-degree of normal subgroups is also locally determined (see \cite{R23}).
Specifically, if $N$ is a normal subgroup of a finite group $G$, $\theta \in \IRR(N)$ is an irreducible complex character of $N$
of degree not divisible by $p$, and $Q$ is a Sylow $p$-subgroup of $N$, then there is a corresponding $\theta^* \in \IRR(\nr[N](Q))$ such that
the character theory of $G$ over $\theta$ is \it{equivalent} to the character theory of $\nr(Q)$ over $\theta^*$.

In \cite{I22}, for a finite solvable group $G$,
I. M. Isaacs constructed a bijection from the set \(\LIN(C)\) of linear characters of a Carter subgroup $C$ of $G$ onto $\HD(G)\subseteq \IRR(G)$, the so called set of head 
characters of $G$.
As shown by G. Navarro in \cite{N22}, this is the ``McKay conjecture” for the formation of nilpotent groups. (For the formation of
$p$-groups, Navarro recovered the original McKay conjecture for solvable groups.)

The purpose of this paper is to show that, above an Isaacs' bijection  of any normal solvable subgroup
$N$ of a finite group $G$, not necessarily solvable,  the character theory of $G$  over a head character
is ``equivalent” to the character theory of the normalizer of a Carter subgroup of $N$ over a linear character of $C$.
This has some consequences for every finite group. For instance, we give a new lower bound for the number \(k(G)\) of conjugacy classes of a finite group $G$ in terms of the Carter subgroups of any normal solvable subgroup of $G$. There does not seem to be a direct group theoretical proof of this fact. 

\medskip

\begin{thmA}\label{thmA:A}
Let $N$ be a normal solvable subgroup of a finite group $G$ and let $C$ be a Carter subgroup of $N$. Then there exists a \((\nr(C)\times \GAL(\Q_{|G|}/\Q))\)-equivariant bijection 
\[{}^*:\HD(N)\rightarrow \textnormal{Lin}(C)\]
such that the character-triples \((G_\theta,N,\theta)\) and \((\nr(C)_{\theta^*},C,\theta^*)\) are isomorphic for every \(\theta\in \HD(N)\). Also, if \(\ce{A}\) denotes the set of irreducible characters of $G$ lying over some member of \(\HD(N)\), then there exists a bijection
\[':\ce{A}\rightarrow \IRR(\nr(C)/C')\]
such that \(\chi(1)/\chi'(1)  = \theta(1)\) for each \(\chi\in \ce{A}\) and \(\theta\in \HD(N)\) lying under \(\chi\). As a consequence, \(\chi(1)/\chi'(1)\) divides \(|N:C|\), and  \(k(G)\geqslant k(\nr(C)/C')\) with equality if and only if \(N\) is abelian. 
\end{thmA}

We notice that, in the case where \(N=G\), Theorem~\hyperref[thmA:A]{A} reproves the main result of \cite{N25}. 

Moreover, when $C$ is a self-normalizing Sylow \(p\)-subgroup of $N$, we show that  Theorem~\hyperref[thmA:A]{A} recovers a very special case of the main result of \cite{R23} without relying on the classification of finite simple groups.

\section{A Strong Character-Triple Isomorphism}

In his landmark paper \cite{I73} (see also \cite[Chapters~7,~8]{CSG}), I. M. Isaacs laid the foundation for the character theory of fully ramified abelian sections. Recall that if \(L\lhd K\) and \(\theta\in \IRR(K)\) lies over \(\varphi\in \IRR(L)\), the characters \(\theta\) and \(\varphi\) are said to be \bf{fully ramified} with respect to \(K/L\) if \(\theta_L = e\varphi\) with \(e^2 = |K:L|\). If, furthermore, \(K/L\) is a normal section of a group $G$ and either \(\theta\) or \(\varphi\) is $G$-invariant, then both of them are, and in this situation, we sometimes say that the section \(K/L\) is fully ramified. 

If \(K/L\) is a fully ramified abelian section of a group $G$, where either \(|K:L|\) or \(|G:K|\) is odd, Isaacs showed  \cite[Theorem~9.1]{I73} the existence of a complement \(H/L\) for \(K/L\) in \(G/L\), a character \(\Psi\) of \(H/L\), and a correspondence \(\IRR(G|\theta)\rightarrow \IRR(H|\varphi)\) defined by the equation \(\chi_H= \Psi \xi\) for \(\chi\in \IRR(G|\theta)\) and \(\xi\in \IRR(H|\varphi)\). 

\medskip

In his expository work \cite{I82}, Isaacs addressed the case where  \(|G:K|\) and \(|K:L|\) may both be even, under the additional condition that \(K/L\) is a \bf{strong section}. An abelian normal section \(K/L\) of a group $G$ is said to be strong if there exists \(N\lhd G\) such that the induced group of automorphisms, \(S=N/\cn[N](K/L)\), satisfies the following conditions. 
\begin{enumerate}
\item \(\cn[K/L](S)=1\). 
\item \(|S|\) and \(|K:L|\) are relatively prime. 
\item $S$ is solvable. 
\end{enumerate}

If \(K/L\) is a strong section of $G$ and \(\theta,\varphi\) are fully ramified with respect to \(K/L\) with \(\theta_L = e \varphi\), Theorem B of \cite{I82} guarantees the existence of a complement \(H/L\) for \(K/L\) in \(G/L\) and a correspondence \(\IRR(H|\varphi)\rightarrow \IRR(G|\theta)\) such that \(\chi(1)=e\xi(1)\) whenever \(\xi\in \IRR(H|\varphi)\) and \(\chi\in \IRR(G|\theta)\) correspond.

\medskip 

Furthermore, we observe from the proofs of \cite[Theorem~9.1]{I73} and \cite[Theorem~B]{I82} that both correspondences define strong character-triple isomorphisms between \((G,\) \(K, \theta)\) and \((H,L,\varphi)\),  where the underlying isomorphism is the natural map \(Kh\mapsto Lh\). (We refer the reader to \cite[Chapter~11]{CSG} and \cite[Chapter~5]{CTN} for the definition and main properties of character-triple isomorphisms. Note that while a character-triple isomorphism is a pair \((\tau,\sigma)\) consisting of a group isomorphism and a bijection on characters, we will slightly abuse notation and use \(\sigma\) to denote both maps. Also, we will write the action of \(\sigma\) exponentially, so that the image of an element $x$ and a character \(\xi\) under \(\sigma\) is \(x^\sigma\) and \(\xi^\sigma\), respectively.) Our aim in this section is to prove that both strong character-triple isomorphisms are compatible, in the sense that they can be chosen to agree in a situation where both are defined. More precisely, we establish the following.

\begin{thm} \label{thm:strongchartripcomp} Assume that \(K/L\) is a strong section of \(G\) and let  \(K\subset N\lhd G\) such that \(|K:L|\) and \(|N:K|\) are relatively prime. Let \(\theta\in \IRR(K)\) and \(\varphi\in \IRR(L)\) be fully ramified with respect to \(K/L\), where \(\varphi\) is $G$-invariant. Then there exist a complement \(H/L\) for \(K/L\) in \(G/L\), an $H$-invariant character \(\Psi\) of \(U/L=(N\cap H)/L\), and a strong character-triple isomorphism \[\sigma:(H,L,\varphi)\rightarrow (G,K,\theta)\] such that the following hold.
\begin{enumerate}
\item The associated isomorphism \(\sigma:H/L\rightarrow G/K\) is the natural map \(Lh\mapsto Kh\).
\item \(\Psi(u) = \pm \sqrt{|\cn[K/L](u)|}\) for all elements \(u\in U\). 
\item For  intermediate subgroups \(L\subset V\subset U\) and \(K\subset W =KV\subset N\), the associated bijection \(\sigma_V:\IRR(V|\varphi)\rightarrow \IRR(W|\theta):\xi\mapsto\xi^{\sigma_V}\) is defined by the equation \((\xi^{\sigma_V})_V = \Psi_V \xi\). 
\end{enumerate}
\end{thm}

Before proving Theorem~\ref{thm:strongchartripcomp}, we check that its conclusions are invariant under strong character-triple isomorphisms. 

\begin{lem}\label{lem:tech}
Assume the hypotheses and notation of Theorem~\ref{thm:strongchartripcomp}.
Let \(\sigma:(G,L,\varphi)\rightarrow (G^\sigma,L^\sigma,\varphi^{\sigma_L})\) be a strong character-triple isomorphism. If  Theorem~\ref{thm:strongchartripcomp} holds for \(G^\sigma\), then it also holds for \(G\). 
\end{lem}

\begin{proof}
Let \(H^\sigma/L^\sigma =(H/L)^\sigma\) be a complement for \(K^\sigma/L^\sigma\) in \(G^\sigma/L^\sigma\) such that there exist an \(H^\sigma\)-invariant character \(\Psi\) of \(U^\sigma/L^\sigma = (N^\sigma\cap H^\sigma)/L^\sigma\), and a strong character-triple isomorphism
\[ \mu:(H^\sigma,L^\sigma,\varphi^{\sigma_L})\rightarrow (G^\sigma,K^\sigma,\theta^{\sigma_K})\]
satisfying the conclusions of Theorem~\ref{thm:strongchartripcomp}.  Let \(\bar{\sigma}:(G,K,\theta)\rightarrow (G^\sigma,K^\sigma,\theta^{\sigma_K})\) be the strong character-triple isomorphism induced by \(\sigma\), so that the underlying isomorphism is defined by \((Kg)^{\bar{\sigma}} = K^\sigma(Lg)^{\sigma}\) and the map on characters
\[\bar{\sigma}:\bigcup_{K\subset X\subset G} \te{Char}(X|\theta)\rightarrow \bigcup_{K^\sigma\subset X^\sigma\subset G^\sigma} \te{Char}(X^\sigma|\theta^{\sigma_K}),\]
where \(\te{Char}(X|\theta)\) is the set of characters \(\xi\) of $X$ such that \(\xi_K\) is a multiple of \(\theta\), 
is simply the restriction of \(\sigma\). 

Consider the composition
\[\nu:(H,L,\varphi)\stackrel{\sigma}{\rightarrow} (H^\sigma,L^\sigma,\varphi^{\sigma_L}) \stackrel{\mu}{\rightarrow} (G^\sigma,K^\sigma,\theta^{\sigma_K}) \stackrel{\bar{\sigma}^{-1}}{\rightarrow }(G,K,\theta),\]
which is itself a strong character-triple isomorphism by \cite[Problem~5.4]{CTN}. It is clear that the underlying isomorphism of \(\nu\) is the natural map \(Lh\mapsto Kh\). Defining \(\Psi^{\sigma^{-1}}\) as the character of \(U/L\) given by \(\Psi^{\sigma^{-1}} (Lu) = \Psi((Lu)^\sigma) \) for \(Lu\in U/L\), we see that it satisfies condition (2) of Theorem~\ref{thm:strongchartripcomp}. 

To check condition (3), let \(L\subset V\subset U\) and \(W=V^\nu   = KV\). We must prove that characters \(\xi\in \IRR(V|\varphi)\) and \(\chi\in \IRR(W|\theta)\) correspond under \(\nu\) if and only if \(\chi_V = (\Psi^{\sigma^{-1}})_V\xi\). 

Suppose first that \(\chi = \xi^{\nu_V} = \xi^{(\sigma\mu\bar{\sigma}^{-1})_V}\). We have
\[ (\xi^{(\sigma\mu\bar{\sigma}^{-1})_V})_V =((\xi^{\sigma_V})^{ \mu_{V^\sigma}})^{(\sigma^{-1}) _{ W^\sigma}})_V\, .\]
Writing
\[ \zeta =\xi^{\sigma_V}\in \IRR(V^\sigma|\varphi^{\sigma_L})\quad \te{and}\quad \eta = \zeta^{\mu_{V^\sigma}}\in \IRR(W^\sigma|\theta^{\sigma_K}),\]
we may apply \cite[Definition~11.23(c)]{CTFG} to obtain 
\[ (\eta^{ (\sigma^{-1})_{W^\sigma}})_V= (\eta_{V^\sigma})^{(\sigma^{-1}) _ {V^\sigma}}\, .\]
Since \(G^\sigma\) satisfies condition (3) of Theorem~\ref{thm:strongchartripcomp} by hypothesis, we have
\[ \eta_{V^\sigma} = \Psi_{V^\sigma} \zeta\, .\]
Applying \cite[Definition~11.23(c,d)]{CTFG}, 
\[ (\Psi_{V^\sigma} \zeta)^{(\sigma^{-1})_{V^\sigma}} = (\Psi_{V^\sigma})^{\sigma^{-1}} \zeta^{(\sigma^{-1})_{V^\sigma}} = (\Psi^{\sigma^{-1}})_V \xi\, .\]
Thus \(\chi_V = (\Psi^{\sigma^{-1}})_V \xi\), as required. 

Conversely, suppose that \(\chi_V = (\Psi^{\sigma^{-1}})_V \xi\). Since
\[ (\chi_V)^{\sigma_V} = (\chi^{\sigma_W})_{V^\sigma}\]
and
\[ ((\Psi^{\sigma^{-1}})_V \xi)^{ \sigma_V}  =((\Psi^{\sigma^{-1}})_V) ^{\sigma} \xi^{\sigma_V}  = \Psi_V \xi^{\sigma_V},\]
we have \((\chi^{\sigma_W})_{V^\sigma} = \Psi_V \xi^{\sigma_V}\). By Theorem~\ref{thm:strongchartripcomp}(3), \(\chi^{\sigma_W}\) is the image of \(\xi^{\sigma_V}\) under \(\mu_{V^\sigma}\). Applying \((\sigma^{-1})_{W^\sigma}\) to the equality \(\chi^{\sigma_W} = (\xi^{\sigma_V})^{\mu_{V^\sigma}}\) we obtain that
\[ \chi = (( \xi^{\sigma_V})^{\mu_{V^\tau}})^{( \sigma^{-1})_{W^\sigma}} = \xi^{\nu_V}\, .\]
This completes the proof. 
\end{proof}

\begin{proof}[Proof of Theorem~\ref{thm:strongchartripcomp}]
By Theorem 11.28 of \cite{CTFG} and Lemma~\ref{lem:tech}, it is no loss to assume that \(L\subset \zn(G)\). Applying \cite[Lemma~11.26]{CTFG} to the canonical projection \(G\rightarrow G/\ker(\varphi)\), we may also assume that \(\varphi\) is faithful. Then we are in the setting of the proof of \cite[Theorem~B]{I82} and on the hypotheses of \cite[Theorem~5.3]{I82}. 

Let \(H/L\) be the complement for \(K/L\) in \(G/L\) defined in the proof \cite[Theorem~B]{I82}. As in the proof of \cite[Theorem~5.3]{I82}, we fix an isomorphism \(\tau:H\rightarrow H^\tau\) and let \(H^\tau\) act on \(K\) via \(k^{h^\tau} = k^h\). Let \(\Gamma = K\rtimes H^\tau\) be the corresponding semidirect product and view \(K\) and \(H^\tau\) as subgroups of \(\Gamma\). Note that since \(L\subset \zn(G)\), we have \(KL^\tau =K\times L^\tau\) and \(LL^\tau = L\times L^\tau\). Thus \(\theta\) and \(\varphi\) extend to \(\hat{\theta} = \theta\times 1_{L^\tau}\in \IRR(KL^\tau)\) and \(\hat{\varphi} = \varphi\times 1_{L^\tau}\in \IRR(LL^\tau)\), respectively. As shown in the final paragraphs of the proof of \cite[Theorem~B]{I82},  \(\theta\) extends to \(G_0 = K\rtimes (H/L)\). Thus \(\hat{\theta}\) extends to \(\Gamma\), and we fix an extension \(\chi\in \IRR(\Gamma)\) of \(\hat{\theta}\). 

\begin{center}
\begin{tikzpicture}[scale =0.8]
 \begin{scope}[rotate=45,scale=1,transform shape,nodes={fill=white,transform
   shape=false}]
	\node (A11) at (0,0) {$1$};
	\node (A12) at (2,0) {$L^\tau$};
	\node (A14) at (6,0) {$H^\tau$};
	
	\node (A21) at (0,2) {$L\, \varphi$};
	\node (A22) at (2,2) {$LL^\tau\, \hat{\varphi}$};
	\node (A25) at (6,2) {$LH^\tau$};	
	\node (A31) at (0,4) {$K\, \theta$};
	\node (A32) at (2,4) {$KL^\tau\, \hat{\theta}$};
	\node (A35) at (6,4) {$\Gamma\, \chi$};
	
	\draw (A11) -- (A12) -- (A14);
	\draw (A21) -- (A22) -- (A25);
	\draw (A31) -- (A32) -- (A35);
	\draw (A11) -- (A21) -- (A31);
	\draw (A12) -- (A22) -- (A32);
	\draw (A14)  --(A25) -- (A35);
\end{scope}
\end{tikzpicture}
\end{center}

\medskip

\noindent \bf{Step 1.} \it{We construct a strong character-triple isomorphism \((H,L,\varphi)\rightarrow (G,K,\theta)\) depending on the extension \(\chi\in \IRR(\Gamma)\), whose associated isomorphism is the natural map \(Lh\mapsto Kh\).}

\medskip

The isomorphism \(\sigma:H/L\rightarrow \Gamma/KL^\tau\) given by \((Lh)^\sigma = (KL^\tau) h^\tau\) defines a strong character-triple isomorphism
\[ \sigma:(H,L,\varphi)\rightarrow (\Gamma,KL^\tau,\varphi^{\sigma_L})\, .\]
For \(L\subset V\subset H\) and \(\xi\in \IRR(V|\varphi)\), the corresponding character \(\xi^\sigma\in \IRR(KV^\tau|\varphi^{\sigma_L})\) is given by \(\xi^\sigma(kv^\tau)=\xi(v)\) for \(k\in K\) and \(v^\tau\in V^\tau\). 

Since \(\varphi^{\sigma_L} = 1_K\times \varphi^\tau\) and \(\hat{\theta} = \theta\times 1_{L^\tau}\), we see that \(\chi_{KL^\tau} \varphi^{\sigma_L} = \hat{\theta} \varphi^{\sigma_L}\) is an irreducible character of \(KL^\tau\). Then \cite[Lemma~11.27]{CTFG} defines a strong character-triple isomorphism
\[\mu:(\Gamma,KL^\tau,\varphi^{\sigma_L}) \rightarrow (\Gamma,KL^\tau,\hat{\theta}\varphi^{\sigma_L})\, .\]
where the associated isomorphism is the identity. Following the previous notation, the character corresponding to \(\xi\in \IRR(V|\varphi)\) under the composition of \(\sigma\) and \(\mu\) is \(\xi^{\sigma\mu} = \chi_{KV^\tau} \xi^\sigma\in \IRR(KV^\tau|\hat{\theta}\varphi^{\sigma_L})\). 

Next, we consider the epimorphism
\[\Gamma\rightarrow G:kh^\tau\mapsto kh\]
and observe that its kernel is the set \(\{l^{-1}l^\tau\,|\, l\in L\}\). Since \(\varphi\) is linear and \(\theta_L = e\varphi\), we have
\[ (\hat{\theta}\varphi^{\sigma_L})(l^{-1} l^\tau) = \hat{\theta}(l^{-1})\varphi(l) = (e\varphi(l)^{-1} ) \varphi(l) = e = (\hat{\theta} \varphi^{\sigma_L})(1),\]
and thus the kernel of the epimorphism is contained in \(\ker(\hat{\theta}\varphi^{\sigma_L})\). Applying again Lemma 11.26 of \cite{CTFG}, we obtain a strong character-triple isomorphism
\[\nu:(\Gamma,KL^\tau,\hat{\theta}\varphi^{\sigma_L})\rightarrow (G,K,(\hat{\theta}\varphi^{\sigma_L})^\nu)\, .\]
The associated isomorphism is  given by \(KL^\tau h^\tau \mapsto Kh\), and \((\hat{\theta}\varphi^{\sigma_L})^\nu\) is defined by 
\[(\hat{\theta}\varphi^{\sigma_L})^\nu(k) = (\hat{\theta}\varphi^{\sigma_L}) (k) = \hat{\theta}(k)=\theta(k)\, .\]
That is, \((\hat{\theta}\varphi^{\sigma_L})^\nu = \theta\). The composition
\[ \sigma \mu \nu : (H,L,\varphi)\stackrel{\sigma}{\rightarrow} (\Gamma,KL^\tau,\varphi^{\sigma_L}) \stackrel{\mu}{\rightarrow} (\Gamma,KL^\tau,\hat{\theta}\varphi^{\sigma_L}) \stackrel{\nu}{\rightarrow} (G,K,\theta)\]
is a strong character-triple isomorphism
with associated isomorphism
\[ Lh\mapsto (KL^\tau) h^\tau \mapsto (KL^\tau)h^\tau\mapsto Kh,\]
so it satisfies condition (1) of the theorem. For  \(L\subset V\subset H\) and \(\xi\in \IRR(V|\varphi)\), note that the corresponding character \(\xi^{\sigma\mu\nu}\in \IRR(KV|\theta)\) is given by \(\xi^{\sigma\mu\nu} (kv)=\chi(kv^\tau)\xi(v)\). 

\medskip

\noindent \bf{Step 2.} \it{We obtain an extension \(\chi\in \IRR(\Gamma)\) of \(\hat{\theta}\) such that the associated strong character-triple isomorphism of Step 1 satisfies conditions (2) and (3).}

\medskip

As in the statement of the theorem, write \(U=N\cap H\). We have that \(\hat{\varphi}\) is \(\Gamma\)-invariant and, since \(\hat{\theta}_{LL^\tau} = e\hat{\varphi}\) with \(e^2 = |K:L|=|KL^\tau:LL^\tau|\), the characters \(\hat{\theta}\) and \(\hat{\varphi}\) are fully ramified with respect to \(KL^\tau/LL^\tau\). Because \(|KU^\tau:KL^\tau| = |U:L|=|N:K|\) and \(|KL^\tau:LL^\tau|\) are relatively prime, we may apply \cite[Theorem~9.1]{I73} to \((KU^\tau,KL^\tau,LL^\tau,\hat{\theta},\hat{\varphi})\). By the conjugacy part of the Schur-Zassenhaus theorem, the conclusions of \cite[Theorem~9.1]{I73} hold for the complement \(LU^\tau/LL^\tau\) of \(KL^\tau/LL^\tau\) in \(KU^\tau/LL^\tau\). Let \(\Psi\) be the canonical character of \(LU^\tau/LL^\tau\) associated with \((KU^\tau,KL^\tau,LL^\tau,\hat{\theta},\hat{\varphi})\) as defined before \cite[Theorem~9.1]{I73}. Since \(\Psi\) is uniquely determined by the \(H^\tau\)-invariant tuple \((KU^\tau, KL^\tau, LL^\tau, \hat{\theta}, \hat{\varphi})\), it follows that \(\Psi\) is \(H^\tau\)-invariant. 

Let \(\zeta = \varphi\times 1_{H^\tau}\in \IRR(LH^\tau| \hat{\varphi})\). By \cite[Theorem~9.1]{I73}, let \(\eta\) be the unique member of \(\IRR(KU^\tau|\hat{\theta})\) such that \(\eta_{LU^\tau} = \Psi \zeta_{LU^\tau}\). By comparing degrees (recall that $\varphi$ is linear), we deduce that \(\eta\) extends \(\hat{\theta}\). Also, since \(\eta\) is uniquely determined by the \(H^\tau\)-invariant characters \(\Psi\) and \(\zeta_{LU^\tau}\), we see that \(\eta\) is invariant in \((KU^\tau) H^\tau = \Gamma\). We claim that \(\eta\) extends to \(\Gamma\). By \cite[Corollary~11.31]{CTFG}, it suffices to prove for every prime $p$ that \(\eta\) extends to  \(P\supset KU^\tau\), where \(P/KU^\tau\) is a Sylow $p$-subgroup of \(\Gamma/KU^\tau\). Because \(|KL^\tau:LL^\tau|=|K:L|\) and \(|KU^\tau:KL^\tau|=|N:K|\) are relatively prime, we distinguish two cases. 

\begin{center}
\begin{tikzpicture}[scale =0.8]
 \begin{scope}[rotate=45,scale=1,transform shape,nodes={fill=white,transform
   shape=false}]
	\node (A11) at (0,0) {$1$};
	\node (A12) at (2,0) {$L^\tau$};
	\node (A13) at (4,0) {$U^\tau$};
	\node (A14) at (6,0) {$H^\tau$};
	
	\node (A21) at (0,2) {$L\, \varphi$};
	\node (A22) at (2,2) {$LL^\tau\, \hat{\varphi}$};
	\node (A23) at (4,2) {$L U^\tau\, \Psi$};
	\node (A25) at (6,2) {$LH^\tau\, \zeta$};	
	\node (A31) at (0,4) {$K\, \theta$};
	\node (A32) at (2,4) {$KL^\tau\, \hat{\theta}, \varphi^{\sigma_L}$};
	\node (A33) at (4,4) {$KU^\tau\, \eta$};
	\node (A35) at (6,4) {$\Gamma\, \chi$};
	
	\draw (A11) -- (A12) -- (A13) -- (A14);
	\draw (A21) -- (A22) -- (A23)  -- (A25);
	\draw (A31) -- (A32) -- (A33)  -- (A35);
	\draw (A11) -- (A21) -- (A31);
	\draw (A12) -- (A22) -- (A32);
	\draw (A13) -- (A23) -- (A33);
	\draw (A14) -- (A25) -- (A35);
\end{scope}
\end{tikzpicture}
\end{center}

Suppose first that \(p\) does not divide \(|KU^\tau:KL^\tau|\). Since \(\eta\) and \(\chi_{KU^\tau}\) extend \(\hat{\theta}\), we apply the Gallagher correspondence \cite[Corollary~6.16]{CTFG} to write \(\eta = \lambda \chi_{KU^\tau}\) for some linear \(\lambda\in \IRR(KU^\tau/KL^\tau)\). Then \(\lambda\) is \(\Gamma\)-invariant because it is uniquely determined by the \(\Gamma\)-invariant characters \(\eta\) and \(\chi_{KU^\tau}\). Since \(|P:KU^\tau|\) and \(|KU^\tau:KL^\tau|\) are relatively prime in this case, \(\lambda\) extends to \(\hat{\lambda}\in \IRR(P/KL^\tau)\) by \cite[Corollary~6.28]{CTFG}. Thus \(\hat{\lambda}\chi_P\in \IRR(P)\) is an extension of \(\eta\), as required. 

 If \(p\) divides \(|KU^\tau:KL^\tau|\), then \(|KL^\tau:LL^\tau|\) and \(|P:KL^\tau|\) are relatively prime, and we apply  \cite[Theorem~9.1]{I73} to \((P,KL^\tau,LL^\tau,\hat{\theta},\hat{\varphi})\). Again by the conjugacy part of the Schur-Zassenhaus theorem, the conclusions of \cite[Theorem~9.1]{I73} hold for the complement \((P\cap LH^\tau)/LL^\tau\) for \(KL^\tau/LL^\tau\) in \(P/LL^\tau\). Letting \(\Psi_0\) be the canonical character of \((P\cap LH^\tau)/LL^\tau\) associated with \((P,KL^\tau,LL^\tau,\hat{\theta},\hat{\varphi})\), we must have \((\Psi_0)_{LU^\tau} = \Psi\) by its uniqueness. By \cite[Theorem~9.1]{I73}, let \(\hat{\eta}\) be the unique member of \(\IRR(P|\hat{\theta})\) such that \(\hat{\eta}_{P\cap LH^\tau} = \Psi_0 \zeta_{P\cap LH^\tau}\). By comparing degrees, we deduce that \(\hat{\eta}\) is an extension of \(\hat{\theta}\). In particular, \(\hat{\eta}_{KU^\tau}\in \IRR(KU^\tau|\hat{\theta})\). 
Since 
  \[ (\hat{\eta}_{KU^\tau})_{LU^\tau} = (\hat{\eta}_{P\cap LH^\tau})_{LU^\tau} = (\Psi_0 \zeta_{P\cap LH^\tau})_{LU^\tau} = (\Psi_0)_{LU^\tau} \zeta_{LU^\tau} = \Psi \zeta_{LU^\tau}\]
and \(\eta\) is the unique member of \(\IRR(KU^\tau|\hat{\theta})\) satisfying \(\eta_{LU^\tau} = \Psi \zeta_{LU^\tau}\), we have  \(\hat{\eta}_{KU^\tau} = \eta\), as wanted. We conclude that \(\eta\) extends to some \(\chi\in \IRR(\Gamma)\). 

\medskip

\noindent \bf{Step 3.} \it{We prove that the strong character-triple isomorphism from Step 1, associated with the extension \(\chi\in\IRR(\Gamma)\) from Step 2, satisfies conditions (2) and (3).}

\medskip

Let \(\Psi'\) be the character of \(U/L\) defined by \(\Psi'(Lu) = \Psi(LL^\tau u^\tau)\). Then it is $H$-invariant and, by \cite[Corollary~6.4]{I73} and the comment after its proof, it satisfies condition (2). 

It remains to check that, for \(L\subset V\subset U\) and \(\xi\in \IRR(V|\varphi)\), the corresponding character \(\xi^{\sigma\mu\nu}\in \IRR(KV|\theta)\) satisfies \((\xi^{\sigma\mu\nu})_V =(\Psi')_V \xi\), and that it is uniquely determined by this equality. For elements \(kv\in KV\), we know that \(\xi^{\sigma\mu\nu}(kv)=\chi(kv^\tau)\xi(v)\) by the last paragraph of Step 1. Applying \cite[Theorem~9.1(d)]{I73} to \((KV^\tau,KL^\tau, LL^\tau, \hat{\theta},\hat{\varphi})\), we have that \(\chi_{KV^\tau}\in \IRR(KV^\tau|\hat{\theta})\) vanishes on elements that are not \(KV^\tau\)-conjugate to elements of \(V^\tau\). Then \(\xi^{\sigma\mu\nu}\) vanishes on elements that are not \(KV\)-conjugate to elements of $V$ and, as a result, \(\xi^{\sigma\mu\nu}\) is determined by its restriction to $V$. Using that 
\[\chi_{LV^\tau} = (\chi_{KU^\tau})_{LV^\tau}=\eta_{LV^\tau} =(\eta_{LU^\tau})_{LV^\tau} =  (\Psi \zeta_{LU^\tau})_{LV^\tau} = \Psi_{LV^\tau} \zeta_{LV^\tau}\]
and \(\zeta = \varphi\times 1_{H^\tau}\), for elements \(v\in V\) we have
\[ \xi^{\sigma\mu\nu}(v)=\chi(v^\tau)\xi(v) = \Psi(v^\tau) \zeta(v^\tau) \xi(v) = \Psi'(v)\xi(v),\]
and thus \((\xi^{\sigma\mu\nu})_V = (\Psi')_V \xi\). Since \(\Psi'(v)=\pm \sqrt{|\cn[K/L](v)|}\neq 0\), we conclude that \(\xi^{\sigma\mu\nu}\) is uniquely determined by the equality \((\xi^{\sigma\mu\nu})_V = (\Psi')_V \xi\). This completes the proof of the theorem. 
\end{proof}

\section{The Inductive Step}

We recall that a Carter subgroup of a solvable group $N$ is a self-normalizing nilpotent subgroup $C$ of $N$, and any two of them are $N$-conjugate. Also, if \(L\) is a normal subgroup of $N$, then \(CL/L\) is a Carter subgroup of \(N/L\) (see \cite{C61}).

\medskip

To establish the inductive step for the proof of Theorem~\hyperref[thmA:A]{A}, we need \cite[Theorem~3.1]{I22}, which we include here for the reader's convenience. 

\begin{thm}\label{thm:3.1}
Let $C$ be a Carter subgroup of a solvable group $N$, and let \(K/L\) be an abelian normal section of $N$ such that \(N=KC\) and \(K\cap LC=L\). The following then hold.
\begin{enumerate}[(a)]
\item If \(\theta\in \IRR_C(K)\), then \(\theta\) lies over a unique member \(\theta'\) of \(\IRR_C(L)\). 
\item If \(\varphi\in \IRR_C(L)\), then \(\varphi\) lies under a unique member \(\tilde{\varphi}\) of \(\IRR_C(K)\). 
\item The maps \(':\IRR_C(K)\rightarrow \IRR_C(L)\) and \(\ \tilde{}:\IRR_C(L)\rightarrow \IRR_C(K)\) are inverse bijections. 
\item If \(\theta\in \IRR_C(K)\) lies over \(\varphi\in \IRR_C(L)\), then \(\theta\) extends to \(N\) if and only if \(\varphi\) extends to \(LC\). 
\end{enumerate}
\end{thm}

\begin{proof}
It follows by Lemma 2.1 and Theorem 3.1 of \cite{I22}. 
\end{proof}

Before stating the next result, we introduce the following notation. If \(H\) is a subgroup of a finite group \(G\) and \(\ce{X}\subset \IRR(H)\), we let 
\[\XT(G|\ce{X})=\{\chi\in \IRR(G)\,|\, \chi_H\in \ce{X}\}\, .\]
If \(\ce{X}\) consists of a single member \(\theta\in \IRR(H)\), we write \(\XT(G|\ce{X})=\XT(G|\theta)\). In the case where \(\ce{X} = \IRR(H)\), we write \(\XT(G|\ce{X}) = \XT(G|H)\).

Recall that if a group $A$ acts on sets $X$ and $Y$, a map $f:X\rightarrow Y$ is $A$\bf{-equivariant} if $f(x^a)=f(x)^a$ for all \(x\in X\) and \(a\in A\). Certainly, compositions of $A$-equivariant maps are $A$-equivariant.

\begin{thm}\label{thm:indbij}
Let \(N\) be a normal solvable subgroup of a finite group $G$, let \(C\) be a Carter subgroup of $N$, and let \(K/L\) be an abelian normal section of $G$ such that \(N=KC\) and \(K\cap LC=L\). Then there exists a \((\nr(C)\times \GAL(\Q_{|G|}/\Q))\)-equivariant bijection 
\[{}^*:\XT(N|K)\rightarrow \XT(LC|L)\]
such that for every \(\chi\in \XT(N|K)\) the characters \(\theta = \chi_K\in \IRR_C(K)\) and \(\varphi =(\chi^*)_L\in \IRR_C(L)\) correspond under the canonical bijection of Theorem~\ref{thm:3.1}. Furthermore, the character-triples \((G_\chi,N,\chi)\) and \(((L\nr(C))_{\chi^*},LC,\chi^*)\) are isomorphic for every \(\chi\in \XT(N|K)\). 
\end{thm}

We will prove Theorem~\ref{thm:indbij} and Theorem~\ref{thm:indbijorbit} by induction on \(|G:L|\). For convenience, if \((G,N,K,L)\) is a tuple as in Theorem~\ref{thm:indbij}, we say that \(r=|G:L|\) is the relevant index. We also write \(\ce{G}=\GAL(\Q_{|G|}/\Q)\), \(U=LC\) and \(H=L\nr(C)\).

\begin{thm}\label{thm:indbijorbit}
Assume the hypotheses and notation of Theorem~\ref{thm:indbij}. Let \(\theta\in \IRR_H(K)\) and \(\varphi\in \IRR_H(L)\) extend to $N$ and $U$, respectively, where \(\varphi\) lies under \(\theta\). Let \(\ce{X}\) and \(\ce{Y}\) be the \(\ce{G}\)-orbits of \(\theta\) and \(\varphi\) in \(\IRR_C(K)\) and \(\IRR_C(L)\), respectively. Assume that Theorem~\ref{thm:indbij} holds whenever the relevant index satisfies \(r\ls |G:L|\). Then there exists an \((H\times \ce{G})\)-equivariant bijection 
\[{}^*:\XT(N|\X)\rightarrow \XT(U|\Y)\]
such that the character-triples \((G_\chi,N,\chi)\) and \((H_{\chi^*},U,\chi^*)\) are isomorphic for every \(\chi\in \XT(N|\X)\). 
\end{thm}

In the case where \(\theta\) and \(\varphi\) are fully ramified with respect to \(K/L\), we will use our strong character-triple isomorphism of Theorem~\ref{thm:strongchartripcomp} to prove Theorem~\ref{thm:indbijorbit}. To this end, we establish the following lemma. Its proof relies on the fact that, if a solvable group $G$ has an abelian nilpotent residual \(K=G^\infty\), then the Carter subgroups of $G$ are precisely the complements for $K$ in $G$ \cite[Lemma~2.1]{N25}.

\begin{lem}\label{lem:conmnorcar}
In the hypotheses of Theorem~\ref{thm:indbij}, the complements for \(K/L\) in \(G/L\) are $G$-conjugate to \(H/L\). 
\end{lem}

\begin{proof}
 By the Frattini argument, we have 
\[G=N\nr(C)=(KC)\nr(C)=K\nr(C)=(KL)\nr(C) = KH\]
and, by Dedekind's lemma,
\[K\cap H=L(K\cap \nr(C)) = L(K\cap C) = K\cap LC=L\, .  \]
Hence, \(H/L\)  complements \(K/L\) in \(G/L\).

Next, observe that \(LC/L\) is a Carter subgroup of \(N/L\) and it is complemented by \(K/L\). Then it is clear that \(H=L\nr(C)\) is the normalizer of \(U=LC\) in $G$ and \(K/L\) is the nilpotent residual \((N/L)^\infty\) of \(N/L\). By  \cite[Lemma~2.1]{N25}, the complements for \(K/L\) in \(N/L\) are the Carter subgroups of \(N/L\).

Now, let \(H_0/L\) be a complement for \(K/L\) in \(G/L\). Then \((H_0\cap N)/L\) is a complement for  \(K/L\) in \(N/L\) and, by the preceding argument, it is a Carter subgroup of \(N/L\). Replacing \(H_0\) by an $N$-conjugate if necessary, we may assume that \(H_0\cap N=LC=U\). Thus \(U=H_0\cap N\lhd H_0\subset \nr(U)=H\). We conclude that \(H_0=H\) and the result follows. 
\end{proof}

To apply Theorem~\ref{thm:strongchartripcomp}, it is also required that \(|K:L|\) and \(|N:K|\) be relatively prime. To handle the non-coprime case, we apply  Theorem~\ref{thm:indbij} to a tuple \((G,N,Q,P)\) with relevant index \(r=|G:P|\ls |G:L|\) together with the following lemma.

\begin{lem}\label{lem:eqextTA}
Let $C$ be a Carter subgroup of a solvable group $N$. Let \(L\lhd K\) be $C$-invariant subgroups of $N$, where \(K/L\) is an abelian \(p\)-group that has no proper $C$-invariant subgroups, and \(K\cap LC=L\). Let \(L\lhd P\subset LC\) be $C$-invariant such that \(P/L\) is an abelian $p$-group, and write \(Q=KP\). Let \(\theta\in \IRR(K)\) lie over \(\varphi\in \IRR(L)\), and let \(\alpha\in \IRR(Q)\) lie over \(\beta\in \IRR(P)\). Suppose that \(\theta\) and \(\alpha\) extend to \(KC=QC\), and that \(\varphi\) and \(\beta\) extend to \(LC=PC\). Then \(\alpha_K = \theta\) if and only if \(\beta_L =\varphi\). 
\end{lem}

\begin{proof}
There is nothing to prove if \(P=L\). We assume that \(P\g L\) and proceed by induction on \(|P:L|\). If \(P/L\) (and consequently \(Q/K\)) have no proper $C$-invariant subgroups, the result follows from Lemmas 5.1 and 5.4 of \cite{I22} (that \(K/L\) is an FPF-section and \(P/L\) is a TA-section follows by Lemmas 2.1 and 2.4 of \cite{I22}, respectively).

Otherwise, let \(L\lhd S\subset P\) be minimal $C$-invariant. Since \(KS/L\) is a $p$-group, 
\[Z/L=\cn[K/L](S)\supset \zn(KS/L)\cap (K/L)\g 1\]
is nontrivial and $C$-invariant. As \(K/L\) contains no proper $C$-invariant subgroups, we must have \(Z=K\), which implies \([K,S]\subset L\subset S\). Thus \(S\lhd T=KS\) and, by Dedekind's lemma, 
\[   T\cap SC = KS\cap SC=S(K\cap SC)=S(K\cap LC)=SL = S\, .\]
%
   	

If \(\alpha_K = \theta\), write \(\gamma = \alpha_T\in \IRR(T)\). Applying Theorem~\ref{thm:3.1}, let \(\delta\) be the unique member of \(\IRR_C(S)\) lying under \(\gamma\). Then \(\delta\) extends to \(SC=PC\) and the inductive hypothesis yields \(\beta_S = \delta\). By \cite[Lemma~5.4]{I22}, \(\delta\) lies over \(\varphi\) and \(\delta_L = \varphi\) by \cite[Lemma~5.1]{I22}. A similar argument shows that \(\alpha_K = \theta\) whenever \(\beta_L = \varphi\), which completes the proof.
\end{proof}

We are now ready to prove Theorem~\ref{thm:indbijorbit}.

\begin{proof}[Proof of Theorem~\ref{thm:indbijorbit}]
If \(L=G\), then \(U=G\) and there is nothing to prove. We assume that \(L\ls G\) and proceed by induction on \(|G:L|\). Let \(':\IRR_C(K)\rightarrow \IRR_C(L)\) be the canonical map of Theorem~\ref{thm:3.1}. Note that \(H\times \ce{G}\) acts on both \(\IRR_C(K)\) and \(\IRR_C(L)\). For \(\alpha\in \IRR_C(K)\), \(\beta\in \IRR_C(L)\), and \(\sigma\in H\times \ce{G}\), since \((\alpha_L)^\sigma =(\alpha^\sigma)_L\) we see that \(\alpha\) lies over \(\beta\) if and only if \(\alpha^\sigma\) lies over \(\beta^\sigma\). Thus the canonical map \('\) is \((H\times \ce{G})\)-equivariant. In particular, \(\Y=\X'\).

Suppose there exists \(M\lhd G\) with \(L\ls M\ls K\). By Theorem~\ref{thm:3.1}, let \(\eta\) be the unique member of \(\IRR_C(M)\) lying under \(\theta\). Again by Theorem~\ref{thm:3.1}, \(\eta\) lies over a unique member \(\varphi_0\) of \(\IRR_C(L)\). Since then \(\varphi_0\) lies under \(\theta\), we see that \(\varphi_0 = \varphi\). Let \(\ce{Z}\) be the \(\ce{G}\)-orbit of \(\eta\) in \(\IRR_C(M)\). Note that \(\eta\) is $H$-invariant, as it is uniquely determined by \(\theta\) and \(\varphi\), and it extends to \(MC\) by Theorem~\ref{thm:3.1}.

It is clear that we can apply the inductive hypothesis to \((G,N,C,K,M,\theta,\eta)\) and \((MH,MC,C,M,L,\eta,\varphi)\) in place of \((G,N,C,K,L,\theta,\varphi)\). We obtain an \(((MH)\times \ce{G})\)-equivariant bijection
 \[{}^{*_1}:\XT(N|\ce{X})\rightarrow \XT(MC|\ce{Z}),\]
an \((H\times (\GAL(\Q_{|MH|}/\Q))\)-equivariant bijection
\[{}^{*_2}:\XT(MC|\ce{Z})\rightarrow \XT(U|\ce{Y}),\]
and character-triple isomorphisms
\[ (G_\chi,N,\chi)\rightarrow ( (MH)_{\chi^{*_1}},MC,\chi^{*_1})\rightarrow (H_{(\chi^{*_1})^{*_2}},U,(\chi^{*_1})^{*_2})\]
for all \(\chi\in \XT(N|\X)\). Since \(H\subset MH\), the map \({}^{*_1}\) is \((H\times \ce{G})\)-equivariant. If \(\sigma\in \ce{G}\) and \(\alpha\in \XT(MC|\ce{Z})\), then since the restriction \(\sigma_0\) of \(\sigma\) to \(\Q_{|MH|}\) is an element of \(\GAL(\Q_{|MH|}/\Q)\), we have 
\[(\alpha^\sigma)^{*_2} = (\alpha^{\sigma_0})^{*_2} = (\alpha^{*_2})^{\sigma_0} = (\alpha^{*_2})^\sigma\, .\]
Thus \({}^{*_2}\) is also \((H\times \ce{G})\)-equivariant and the composition
\[\XT(N|\ce{X})\stackrel{*_1}{\rightarrow} \XT(MC|\ce{Z})\stackrel{*_2}{\rightarrow} \XT(U|\ce{Y})\]
is an \((H\times \ce{G})\)-equivariant bijection with the required properties.

We may thus assume that \(K/L\) is a chief factor of $G$. Since \(\varphi\) is \(H\)-invariant,  one of the following occurs by the Going-Down theorem  \cite[Theorem~6.18]{CTFG} (or \cite[Corollary~7.4]{CSG}).
\begin{enumerate}[(1)]
\item \(\varphi^K = \theta\) and \(G_\varphi = H\). 
\item \(\theta_L = \varphi\). 
\item \(\theta\) and \(\varphi\) are fully ramified with respect to \(K/L\). 
\end{enumerate}

Suppose first that \(\varphi^K = \theta\). By the Clifford correspondence \cite[Theorem~6.11]{CTFG}, induction defines a bijection \(\IRR(U|\varphi)\rightarrow \IRR(N|\varphi)\). By Gallagher's correspondence \cite[Corollary~6.17]{CTFG}, 
\[ |\IRR(N|\theta)| = |\IRR(N/K)| = |\IRR(U/L)| = |\IRR(U|\varphi)|\, .\]
Since \(\IRR(N|\theta)\subset \IRR(N|\varphi)\),  equality holds and, comparing degrees, we obtain that induction defines an \(H\)-equivariant bijection \(\XT(U|\varphi)\rightarrow \XT(N|\theta)\). It is then clear that induction defines an \((H\times \ce{G})\)-equivariant bijection \({}^*:\XT(U|\Y)\rightarrow \XT(N|\X)\). Moreover, by using \cite[Problems~5.2,5.3]{CTFG}, we deduce that  induction defines a character-triple isomorphism \((H_{\xi},U,\xi)\rightarrow (G_{\xi^*},N,\xi^*)\) for every \(\xi\in \XT(U|\Y)\).

Suppose now that \(\theta_L=\varphi\), so that both \(\theta\) and \(\varphi\) are $G$-invariant. Applying \cite[Lemma~2.11(a)]{CSG},  restriction defines a bijection \(\IRR(N|\theta)\rightarrow \IRR(U|\varphi)\). By comparing degrees, it is clear that restriction defines an $H$-equivariant bijection \(\XT(N|\theta)\rightarrow \XT(U|\varphi)\) and, therefore, it defines an \((H\times \ce{G})\)-equivariant bijection \({}^*:\XT(N|\X)\rightarrow \XT(U|\Y)\). Furthermore, restriction defines a character-triple isomorphism \((G_\chi,N,\chi)\) \(\rightarrow (H_{\chi^*},U,\chi^*)\) for every \(\chi\in\XT(N|\X)\).

We are left with the case where \(\theta\) and \(\varphi\) are fully ramified with respect to \(K/L\). Let $p$ be the unique prime divisor of \(K/L\) and assume first that \(p\) divides \(|N:K|=|U:L|\). Let \(P/L  = \op_p(\zn(U/L))\lhd H/L\). Because \(U/L\) is nilpotent, \(P/L\) is nontrivial. Then, since \(K/L\) is a normal subgroup of the $p$-group \(Q/L=KP/L\), 
\[Z/L=\cn[K/L](P)\supset \zn(KP/L)\cap (K/L)\g 1\, .\]
 Also, \(Z\lhd K\) because \(K/L\) is abelian, and $H$ normalizes $Z$. It follows that \(Z\lhd KH=G\), which forces \(Z=K\) since \(K/L\) is a chief factor of $G$. Thus \(P/L\) centralizes \(K/L\) and, as result,  \([P,K]\subset L\subset P\). Since \(P\lhd H\), we see that \(P\lhd KH = G\) and \(Q=KP\lhd G\).

Now, by Dedekind's lemma,
\[  Q\cap PC= Q\cap U=KP\cap U=P(K\cap U)=PL=P\, .\]
The tuple \((G,N,Q,P)\) is, therefore, in the hypotheses of Theorem~\ref{thm:indbij}. Since the relevant  index is \(r=|G:P|\ls |G:L|\), there exists by assumption an \((H\times \ce{G})\)-equivariant bijection \({}^*:\XT(N|Q)\rightarrow \XT(U|P)\) satisfying the conditions of Theorem~\ref{thm:indbij}. By Lemma ~\ref{lem:eqextTA}, a member \(\chi\) of \(\XT(N|Q)\) lies in \(\XT(N|\theta)\) if and only if \(\chi^*\in \XT(U|P)\) lies in \(\XT(U|\varphi)\). Since the sets \(\XT(N|\X)\) and \(\XT(U|\Y)\) are \((H\times \ce{G})\)-invariant,  the map \({}^*\) defines an \((H\times \ce{G})\)-equivariant bijection \(\XT(N|\X)\rightarrow \XT(U|\Y)\). Furthermore, by Theorem~\ref{thm:indbij}, the character-triples \(( G_\chi,N,\chi)\) and \((H_{\chi^*},U,\chi^*)\) are isomorphic for every \(\chi\in \XT(N|\X)\).

Finally, we assume that \(|K:L|\) and \(|N:K|\) are relatively prime. Since \(\cn[K/L](C)=1\) by \cite[Lemma~2.1]{I22}, we see that \(\cn[N](K/L)\ls N\). Let \(S=M/\cn[N](K/L)\subset N/\cn[N](K/L)\) be a chief factor of $G$. Then it is clear that $S$ is solvable and \(\cn[K/L](S)=1\). Hence, \(K/L\) is a strong section of $G$ and we are in the hypotheses of Theorem~\ref{thm:strongchartripcomp}. Moreover, by Lemma~\ref{lem:conmnorcar} the conclusions of Theorem~\ref{thm:strongchartripcomp} apply to the complement \(H/L\). 



Let \(\Psi\) be the canonical character of \(U/L\) associated with \(\varphi\) and \(\theta\). Since \(\Psi\) is rational-valued, it is also the canonical character of \(U/L\) associated with \(\varphi^\sigma\) and \(\theta^\sigma\) for any \(\sigma\in\ce{G}\). Then a strong character-triple isomorphism \((H,L,\varphi^\sigma)\rightarrow (G,L,\theta^\sigma)\)  in the conditions of Theorem~\ref{thm:strongchartripcomp} defines $H$-equivariant bijections \(\XT(N|\theta^\sigma)\rightarrow\XT(U|\varphi^\sigma)\) for each \(\sigma\in\ce{G}\). 

By piecing these maps together, we obtain an \(H\)-equivariant bijection \({}^*:\XT(N|\ce{X})\) \(\rightarrow \XT(U|\ce{Y})\) where the character triples \((G_\chi,N,\chi)\) and \((H_{\chi^*},U,\chi^*)\) are isomorphic for every \(\chi\in \XT(N|\X)\).  To see that the bijection \({}^*\) is also \(\ce{G}\)-equivariant, let \(\chi\in \XT(N|\theta)\) and \(\sigma\in \ce{G}\). We have
\[ \Psi (\chi^*)^\sigma = (\Psi \chi^*)^\sigma  = (\chi_U)^\sigma = (\chi^\sigma)_U = \Psi (\chi^\sigma)^*\]
and, since \(\Psi\) is nonvanishing, \((\chi^*)^\sigma = (\chi^\sigma)^*\), as desired. This completes the proof of the theorem. 
\end{proof}

\begin{proof}[Proof of Theorem~\ref{thm:indbij}]
There is nothing to prove if \(L=G\), so we assume that \(L\ls G\) and proceed by induction on \(|G:L|\). As before, we let \(':\IRR_C(K)\rightarrow \IRR_C(L)\) be the canonical map of Theorem~\ref{thm:3.1}. Let 
\[\ce{A} = \{\chi_K\,|\, \chi\in \XT(N|K)\}\quad \te{and}\quad \ce{B} = \{\xi_L\,|\, \xi\in \XT(U|L)\},\]
and note that \(\ce{B} = \ce{A}'\). 

Let \(\ce{X}\subset \ce{A}\) be a set of representatives for the \((H\times \ce{G})\)-orbits on \(\ce{A}\). For each \(\theta\in \ce{X}\), let \(\OR_\theta\) and \(\ce{P}_\theta\) be the \(\ce{G}\)-orbit and the \((H\times \ce{G}\))-orbit of \(\theta\) in \(\IRR_C(K)\), respectively. Then \(\OR'_\theta\) and \(\ce{P}_\theta'\) are the corresponding orbits of \(\theta'\) in \(\IRR_C(L)\). 

Fix \(\theta\in \ce{X}\) and let \(H_\theta=H_{\theta'}\) be the stabilizer of \(\theta\) in $H$. Applying Theorem~\ref{thm:indbijorbit}, let 
\[\pi_\theta:\XT(N|\OR_\theta)\rightarrow \XT(U|\OR_{\theta}')\]
be an \((H_\theta\times \ce{G})\)-equivariant bijection satisfying the properties of the theorem. For \(\chi\in \XT(N|\OR_\theta)\) and \(h\in H\),  set \[(\chi^h)^{\sigma_\theta}= (\chi^{\pi_\theta})^h\, .\] It is routine to check that \(\sigma_\theta:\XT(N|\ce{P}_\theta)\rightarrow \XT(U|\ce{P}_{\theta}')\) is a well-defined \((H\times \ce{G})\)-equivariant bijection. By piecing these maps together for all \(\theta\in \ce{X}\), we obtain an \((H\times \ce{G})\)-equivariant bijection \({}^*:\XT(N|K)\rightarrow \XT(U|L)\) with the required properties. Furthermore,  the character-triples \((G_\theta,K,\theta)\) and \((H_{\theta'},L,\theta')\) are isomorphic by Theorem~\ref{thm:indbijorbit}. 
\end{proof}

\section{Proof of Theorem~\hyperref[thmA:A]{A}}

Before presenting the proof of Theorem~\hyperref[thmA:A]{A}, we briefly review the definition of the head characters. Let $N$ be a finite solvable group and let $C$ be a Carter subgroup of $N$. We define the set of head characters \(\HD(N)\subset \IRR(N)\) by induction on \(|N|\).

If \(N=C\), we set \(\HD(N) = \LIN(N)\). Otherwise, let \(K=N^\infty\ls N\) be the nilpotent residual of $N$ and let \(L=K'\ls K\). If \(LC=N\), then \(N/L\) would be nilpotent contradicting the minimality of \(K\). Thus \(U=LC\ls N\) and, by induction,  the set \(\HD(U)\) is defined. Let \(':\IRR_C(K)\rightarrow \IRR_C(L)\) be the canonical map of Theorem~\ref{thm:3.1}. By \cite[Lemma~5.6]{I22},  \(\HD(N)\) is the set of characters \(\theta\in \IRR(N)\) such that \(\theta_K\in \IRR(K)\) and the corresponding character \((\theta_K)'\in \IRR_C(L)\) lies under some member of \(\HD(U)\). 

Using this definition and inducting on the order of $N$, one can show that the set \(\HD(N)\) is invariant under group automorphisms and Galois automorphisms. For more properties of head characters, see  \cite{I22,N22,N25,FGS26}.

We are now ready to prove Theorem~\hyperref[thmA:A]{A}, which we restate below for the reader's convenience. 

\begin{thm}
Let $N$ be a normal solvable subgroup of a finite group $G$ and let $C$ be a Carter subgroup of $N$. Then there exists a \((\nr(C)\times \GAL(\Q_{|G|}/\Q))\)-equivariant bijection 
\[{}^*:\HD(N)\rightarrow \textnormal{Lin}(C)\]
such that the character-triples \((G_\theta,N,\theta)\) and \((\nr(C)_{\theta^*},C,\theta^*)\) are isomorphic for every \(\theta\in \HD(N)\). Also, if \(\ce{A}\) denotes the set of irreducible characters of $G$ lying over some member of \(\HD(N)\), then there exists a bijection
\[':\ce{A}\rightarrow \IRR(\nr(C)/C')\]
such that \(\chi(1)/\chi'(1)  = \theta(1)\) for each \(\chi\in \ce{A}\) and \(\theta\in \HD(N)\) lying under \(\chi\). As a consequence, \(\chi(1)/\chi'(1)\) divides \(|N:C|\), and  \(k(G)\geqslant k(\nr(C)/C')\) with equality if and only if \(N\) is abelian. 
\end{thm}

\begin{proof}[{Proof of Theorem~\hyperref[thmA:A]{A}}]
There is nothing to prove if \(C=N\). We assume  \(C\ls N\) and proceed by induction on \(|N:C|\). Let \(K=N^\infty\) be the nilpotent residual of $N$ and let \(L=K'\ls K\). If \(LC=N\), then \(N/L\) would be nilpotent contradicting the minimality of \(K\). Thus \(U=LC\ls N\). Writing \(H=L\nr(C)\) and \(\ce{G} = \GAL(\Q_{|G|}/\Q)\), the inductive hypothesis yields an \((H\times \ce{G})\)-equivariant bijection 
\[{}^{*_1}:\HD(U)\rightarrow \te{Lin}(C)\]
satisfying the conditions of Theorem~\hyperref[thmA:A]{A}. 

Now, \(K/L\) is the nilpotent residual \((N/L)^\infty\) of \(N/L\) and it is abelian. Since \(U/L\) is a Carter subgroup of \(N/L\), \cite[Lemma~2.1]{N25} yields \(K\cap U = L\) and we are in the hypotheses of Theorem~\ref{thm:indbij}. Let 
\[{}^{*_2}:\XT (N|K)\rightarrow \XT(U|L)\]
be an \((H\times \ce{G})\)-equivariant bijection in the conditions of Theorem~\ref{thm:indbij}. By Lemma 2.1 of \cite{I22}, \(K/L\) is an FPF section, and by \cite[Lemma~5.6]{I22}, \({}^{*_2}\) defines a bijection
\[ \HD(N)\rightarrow \HD(U)\, .\]
It is clear that the composition 
\[{}^*: \HD(N)\stackrel{{}^{*_2}}{\rightarrow} \HD(U) \stackrel{{}^{*_1}}{\rightarrow} \te{Lin}(C)\]
is an \((\nr(C)\times \ce{G})\)-equivariant bijection such that the character-triples \((G_\theta,N,\theta)\) and \((\nr(C)_{\theta^*},C,\theta^*)\) are isomorphic for every \(\theta\in \HD(N)\), as required. 

Now,  let \(\X\) be a set of representatives for the $\nr(C)$-orbits on \(\HD(N)\), so that \(\X^*\) is a set of representatives for the \(\nr(C)\)-orbits on \(\LIN(C)\). 

Fix \(\theta\in \X\). Consider a bijection 
\[ \IRR(G|\theta)\rightarrow \IRR(G_\theta|\theta)\rightarrow \IRR(\nr(C)_{\theta^*}|\theta^*)\rightarrow \IRR(\nr(C)|\theta^*),\]
where the first and last maps are given by the Clifford correspondence, and the middle map
\[ \IRR(G_\theta|\theta)\rightarrow \IRR(\nr(C)_{\theta^*}|\theta^*)\]
is a bijection induced by a character-triple isomorphism \((G_\theta,N,\theta)\rightarrow (\nr(C)_{\theta^*},C,\theta^*)\), which exists by the first part of the proof. 

Given \(\chi\in \IRR(G|\theta)\), let \(\xi\in \IRR(G_\theta|\theta)\) be its Clifford correspondent, let \(\nu\in \IRR(\nr(C)_{\theta^*}|\theta^*)\) be the correspondent of \(\xi\) under the character-triple isomorphism, and write \(\mu = \nu^{\nr(C)}\in \IRR(\nr(C)|\theta^*)\). Since the map \({}^*\) is $\nr(C)$-equivariant, we have
\[ |G:G_\theta| = |\nr(C):\nr(C)_{\theta^*}|\]
and, by \cite[Lemma~11.24]{CTFG}, 
\[ \frac{\chi(1)}{\mu(1)} = \frac{ |G:G_\theta| \xi(1)}{ |\nr(C):\nr(C)_{\theta^*}| \nu(1)} = \theta(1) \frac{\xi(1)/\theta(1)}{\nu(1)/\theta^*(1)}  =\theta(1)\, .\]
By \cite[Theorem~A]{I22}, \(\chi(1)/\mu(1)=\theta(1)\) divides \(|N:C|\). 

By piecing these maps together, we obtain a bijection
\[ ':\ce{A}\rightarrow \IRR(\nr(C)/C')\]
with the required property. We conclude that
\[ |\IRR(\nr(C)/C')| = |\ce{A}| \leqslant |\IRR(G)|\, .\]
In other words, \(k(G)\geqslant k(\nr(C)/C')\), as desired. 

Finally, assume that \(k(G)=k(\nr(C)/C')\). Then \(\ce{A}=\IRR(G)\) and, hence, \(\HD(N)= \IRR(N)\).  By \cite[Theorem~B]{N22}, $N$ is abelian. This completes the proof. 
\end{proof}

As discussed in the introduction, we obtain the following version of the main result of \cite{R23} in the special case where \(C\) is a self-normalizing Sylow \(p\)-subgroup of $N$.

\begin{cor}
Let \(N\) be a normal solvable subgroup of a finite group $G$ and let $Q$ be a self-normalizing Sylow \(p\)-subgroup of $N$. Then there exists a \((\nr(Q)\times \ce{G})\)-equivariant bijection
\[  {}^* : \IRR_{p'}(N)\rightarrow \textnormal{Lin}(Q)\]
such that the character-triples \((G_\theta,N,\theta)\) and \((\nr(Q)_{\theta^*},Q,\theta^*)\) are isomorphic for every \(\theta\in \IRR_{p'}(N)\). Also, if \(\ce{A}\) denotes the set of irreducible characters of \(G\) lying over some member of \(\IRR_{p'}(N)\), then there exists a bijection
\[ ':\ce{A}\rightarrow \IRR(\nr(Q)/Q')\]
such that \(\chi(1)_p = \chi'(1)_p\) for every \(\chi\in \ce{A}\).  In particular, 
\[|\IRR_{p'}(G)| = |\IRR_{p'}(\nr(Q))|\, .\]
\end{cor}

\begin{proof}
It suffices by Theorem~\hyperref[thmA:A]{A} to prove that \(\HD(N)=\IRR_{p'}(N)\). 

If \(\theta\in \HD(N)\), then  \(\theta(1)\) divides the \(p'\)-number \(|N:Q|\) by \cite[Theorem~A]{I22}. Thus \(\theta\in \IRR_{p'}(N)\) and \(\HD(N)\subset \IRR_{p'}(N)\). Since 
\[|\HD(N)| = |\te{Lin}(Q)| =|\IRR_{p'}(\nr[N](Q))|  = |\IRR_{p'}(N)|,\]
where the last equality follows by the McKay correspondence for solvable groups, the equality \(\HD(N)=\IRR_{p'}(N)\) follows. 
\end{proof}

In fact, we can obtain a much stronger result by using Navarro's canonical bijection (\cite[Theorem~9.4]{CTN})
\[ {}^*:\IRR_{p'}(N)\rightarrow \LIN(Q),\]
where the correspondent  \(\theta^*\) of any \(\theta\in \IRR_{p'}(N)\) is the unique linear constituent of  \(\theta_Q\), and  we have \([\theta_Q,\theta^*] = 1\). Navarro's bijection holds for $p$-solvable groups and, by \cite[Corollary~B]{V14}, it also holds for \(p=3\). Moreover, we note that for \(p\g 3\), every finite group $G$ with a self-normalizing Sylow \(p\)-subgroup is solvable by the main result of \cite{G04}.

\begin{thm}
Let \(N\) be a normal subgroup of a finite group $G$ and let \(Q\) be a self-normalizing Sylow $p$-subgroup of $N$. Assume that \(N\) is $p$-solvable or that \(p=3\).  Let \(\ce{A}\) be the set of irreducible characters of $G$ lying over some member of \(\IRR_{p'}(N)\). Then there exists a natural correspondence
\[ ' :\ce{A}\rightarrow \IRR(\nr(Q)/Q')\]
given by 
\[ \chi_{\nr(Q)} = \chi' + \Delta,\]
where \(\chi'\) is the unique irreducible constituent of \(\chi_{\nr(Q)}\) that lies in \(\IRR(\nr(Q)/Q')\). Furthermore, \(\chi(1)_p = \chi'(1)_p\) for every \(\chi\in \ce{A}\). In particular, 
\[|\IRR_{p'}(G)| = |\IRR_{p'}(\nr(Q))|\, .\]
\end{thm}

\begin{proof}
Write \(H=\nr(Q)\) and let \(\X\) be a set of representatives for the \(H\)-orbits on \(\IRR_{p'}(N)\), so that \(\X^*\) is a set  of representatives for the \(H\)-orbits on \(\LIN(Q)\). Fix \(\theta\in \X\). We work by induction on \(|G|\) to prove that there exists a bijection
\[':\IRR(G|\theta)\rightarrow \IRR(H|\theta^*)\]
satisfying the conditions of the theorem.

Suppose first that \(G_\theta\ls G\). By the inductive hypothesis, there exists a bijection
\[':\IRR(G_\theta|\theta)\rightarrow \IRR(H_{\theta^*}|\theta^*)\]
with the required properties. 
We claim that the composition
\[ \IRR(G|\theta)\rightarrow \IRR(G_\theta|\theta)\stackrel{'}{\rightarrow} \IRR(H_{\theta^*}|\theta^*)\rightarrow \IRR(H|\theta^*)\]
is the desired bijection, where the first and last maps are given by the Clifford correspondence. Let \(\xi\in \IRR(G_\theta|\theta)\) and write
\[ \xi_{H_{\theta^*}} = \nu + \Delta\]
where \(\nu\in \IRR(H_{\theta^*}/Q')\) and no irreducible constituent of \(\Delta\) lies in \(\IRR(H_{\theta^*}/Q')\). Since the map \({}^*\) is $H$-equivariant, we have \(G_\theta\cap H = H_\theta = H_{\theta^*}\) and hence
\[ \chi_H = (\xi^G)_H = (\xi_{H_{\theta^*}})^H = \nu^H  + \Delta^H\]
where \(\mu = \nu^H\in \IRR(H/Q')\) by Clifford's correspondence. We see that no irreducible constituent of \(\Delta^H\) lies in \(\IRR(H/Q')\). Indeed, if \(\psi\in \IRR(H)\) is an irreducible constituent of \(\Delta^H\), then by Frobenius reciprocity there exists an irreducible constituent \(\eta\in \IRR(H_{\theta^*})\) of \(\Delta\) lying under \(\psi\). By Clifford's theorem \cite[Theorem~6.2]{CTFG}, the irreducible constituents of \(\psi_Q\) are $H$-conjugates of irreducible constituents of \(\eta_Q\) and, therefore, the irreducible constituents of \(\psi_Q\) are not linear, as required. Also, 
\[ \chi(1)_p = |G:G_\theta|_p \xi(1)_p = |H:H_{\theta^*}|_p \nu(1)_p = \mu(1)_p\]
and the claim holds in this case. 

Otherwise, both \(\theta\) and \(\theta^*\) are \(H\)-invariant. By \cite[Lemma~6.8]{CTN}, there exists a bijection
\[':\IRR(G|\theta)\rightarrow \IRR(H|\theta^*)\]
given by 
\[\chi_H = \chi' + \Delta,\]
where \(\chi'\in \IRR(H|\theta^*)\) and no irreducible constituent of \(\Delta\) lies in \(\IRR(H|\theta^*)\). Since  then
\[ (\chi')_Q + \Delta_Q = \chi_Q = [\chi_N,\theta] \theta_Q,\]
we deduce that no irreducible constituent of \(\Delta_Q\) is linear. By Part (c) of \cite[Lemma~6.8]{CTN}, we also have 
\[ \chi(1)_p = \frac{\chi(1)_p}{\theta(1)_p} = \frac{\chi'(1)_p}{\theta^*(1)_p} = \chi'(1)_p\, .\] 
We obtain the required correspondence by piecing together these maps for \(\theta\in \X\). 
\end{proof}

\printbibliography

\end{document}